\documentclass[11pt]{amsart}
\usepackage{fullpage}
\usepackage{amsmath, amssymb, amsthm}
\usepackage[foot]{amsaddr}

\usepackage{graphicx}
\usepackage{color}
\usepackage[colorlinks=true,linkcolor=red,citecolor=blue,urlcolor=cyan]{hyperref}

\usepackage{multirow}
\usepackage{multicol}

\newtheorem{thm}{Theorem}[section]
\newtheorem{no}[thm]{Numerical Observation}
\newtheorem{lem}[thm]{Lemma}
\theoremstyle{remark}
\newtheorem{rem}[thm]{Remark}

\newcommand{\ie}{{\it i.e.}}
\newcommand{\eg}{{\it e.g.}}

\newcommand{\ko}[1]{\text{\color{red}#1}}

\usepackage[backend=biber,date=year,giveninits=true,sorting=nyt,natbib=true,maxcitenames=2,maxbibnames=10,url=false,doi=true,backref=false]{biblatex}
\renewbibmacro{in:}{\ifentrytype{article}{}{\printtext{\bibstring{in}\intitlepunct}}}

\begin{document}
\title{Flat tori with large Laplacian \\ eigenvalues in dimensions up to eight}
\thanks{C. Kao acknowledges partial support from NSF DMS 1818948. B. Osting acknowledges partial support from NSF DMS 17-52202.}

\author{Chiu-Yen Kao}
\address{Department of Mathematical Sciences, Claremont McKenna College, Claremont, CA }
\email{ckao@cmc.edu}

\author{Braxton Osting}
\address{Department of Mathematics, University of Utah, Salt Lake City, UT}
\email{osting@math.utah.edu}

\author{Jackson C. Turner}
\address{Department of Applied Physics and Applied Mathematics, Columbia University, New York City, NY}
\email{jt3287@columbia.edu }



\date{\today}

\begin{abstract} 
We consider the optimization problem of maximizing the $k$-th Laplacian eigenvalue, $\lambda_{k}$, over flat $d$-dimensional tori of fixed volume. For $k=1$, this problem is equivalent to the densest lattice sphere packing problem. 
For larger $k$, this is equivalent to the NP-hard problem of finding the $d$-dimensional (dual) lattice with longest $k$-th shortest lattice vector. As a result of extensive computations, for $d \leq 8$, we obtain a sequence of flat tori, $T_{k,d}$, each of volume one, such that the $k$-th Laplacian eigenvalue of $T_{k,d}$ is very large; 
for each (finite) $k$ the $k$-th eigenvalue exceeds the value in (the $k\to \infty$ asymptotic) Weyl's law by a factor between 1.54 and 2.01, depending on the dimension. Stationarity conditions are derived and numerically verified for $T_{k,d}$ and we describe the degeneration of the tori as $k \to \infty$.
\end{abstract}

\maketitle 

\section{Introduction}
\label{s:Intro}
Consider the $d$-dimensional lattice $\Gamma_B := B \mathbb Z^d$ generated by the basis matrix $B \in GL(d,\mathbb R)$ and the $d$-dimensional flat torus $T_B := \mathbb R^d / \Gamma_B$. 
The volume of $T_B$ is given by 
$\mathrm{vol}(T_B) = | \det B| $.
Each eigenpair, $(\lambda,\psi)$, of the Laplacian, $-\Delta$, on $T_B$ corresponds to an element of the dual lattice, $\Gamma_B^*  = B^{-t} \mathbb Z^d = \Gamma_{B^{-t}}$:
$$
\lambda = 4 \pi^2 \|w\|^2, 
\qquad 
\psi(x) = e^{2 \pi i \langle x, w \rangle}, 
\qquad 
\forall x \in T_B, \ w \in \Gamma_B^*.
$$ 
The multiplicity of each non-zero eigenvalue is even since $w \in \Gamma_B^*$ and $-w$ correspond to the same eigenvalue. 
It follows that the eigenvalues of $-\Delta$ on $T_B$,
enumerated in increasing order including multiplicity, 
$$
0 = \lambda_0 < \lambda_1 = \lambda_2 \leq \lambda_3 = \lambda_4 \leq \cdots,
$$
 are characterized by the Courant-Fischer formulae, 
\begin{equation} \label{e:CF}
\lambda_k(T_B)
= \min_{E \in \mathbb Z^d_{k+1}} \ \max_{v \in  E} \ 4 \pi^2 \| B^{-t} v \|^2,
\end{equation}
where $\mathbb Z^d_k := \{ E \subset \mathbb Z^d \colon  |E| = k \}$. 
Since the multiplicity is even, throughout this manuscript, it will be convenient to use the notation $\kappa :=2 \left\lceil \frac{k}{2} \right\rceil$, where $\left\lceil \cdot \right\rceil$ is the ceiling function. 

For $k \in \mathbb N$, define the \emph{volume-normalized Laplacian eigenvalue}, $\Lambda_k \colon GL(d,\mathbb R) \to \mathbb R$, by 
\begin{equation} \label{e:Lambda}
\Lambda_{k,d}(B) = \lambda_k(T_B) \cdot \mathrm{vol}(T_B)^{\frac 2 d}. 
\end{equation}
The volume-normalized eigenvalues are scale invariant in the sense that $\Lambda_{k,d} (\alpha B ) = \Lambda_{k,d} (B)$ for all $\alpha \in \mathbb R \setminus \{0\}$. 
Weyl's law states that for any  $B \in GL(d,\mathbb R)$, 
\begin{equation} \label{e:Weyl}
\textstyle 
\Lambda_{k,d}(B) \sim  g_d \  \pi^2  \   k^{\frac{2}{d}} ,
\qquad \qquad \textrm{as } k \to \infty,
\end{equation}
where $g_d = 4 (\omega_d)^{- \frac{2}{d}}$ and $\omega_d = \frac{\pi^{\frac d 2}}{ \Gamma\left( \frac d 2 + 1\right)}$ is the volume of the unit ball in $\mathbb R^d$.

\bigskip

In this work, for fixed $k,d \in \mathbb N$, we consider the eigenvalue optimization problem 
\begin{equation} \label{e:OptProb}
\Lambda_{k,d}^\star = \max_{B\in GL(d,\mathbb R)} \ \Lambda_{k,d} (B).
\end{equation}
The existence of a matrix  $B^\star$ attaining the maximum in \eqref{e:OptProb} was proven in \cite[Theorem~1.1]{Lagac__2019}. 
The tori $T_A$ and $T_B$ are isometric if and only if $A$ and $B$ are equivalent in 
$$
O(d,\mathbb R) \setminus GL(d,\mathbb R) / GL(d,\mathbb Z).
$$
Here, $O(d,\mathbb R) $ is the group of orthogonal matrices and $GL(d,\mathbb Z)$ is the group of unimodular matrices.
Since the Laplacian spectrum is preserved by isometry, it follows that the solution to the optimization problem in \eqref{e:OptProb} is not unique. 
Minkowski's first fundamental theorem implies that $\Lambda_{1,d}^\star \leq 4 \pi^2 d$; see, \eg, Theorem 22.1 and Corollary 22.1 in \cite{Gruber2007}. Together with the Courant-Fischer formula \eqref{e:CF},  this result  implies that
$\Lambda_{k,d}^\star \leq d \pi^2  \kappa^2$.

For general $d$ and $k$, the maximizer in \eqref{e:OptProb} is unknown. In  dimension $d=1$, it is easy to see that  $\Lambda_{k,1}^\star = \pi^2 \kappa^2$. In  dimension $d=2$, it was shown by M. Berger that $\Lambda^\star_{1,2} = \frac{8 \pi^2}{\sqrt{3}}$ is attained by the basis 
$B_{1,2} = 
\frac{1}{2} \begin{pmatrix}
2 & 1 \\
0 & \sqrt{3}
\end{pmatrix}$, 
which generates the equilateral torus \cite{Berger1973}. It was shown in \cite{KaoLaiOsting} that for $k\geq 1$, 
\begin{equation} 
\label{e:Bk2}
B_{k,2} = 
\frac{1}{2} \begin{pmatrix}
2 & 1  \\
0 & \sqrt{ \kappa^2 - 1}  
\end{pmatrix}
\end{equation}
is a local maximum with value 
$ \textstyle
\Lambda_{k,2} = \frac{2 \pi^2 \kappa^2} {\sqrt{\kappa^2 -1 }}$.
It is shown that this lattice is globally optimal for $k=1,2,3,4$. For each $k$, the corresponding eigenvalue has multiplicity 6 and,   
as $k\to \infty$, the flat tori generated by these bases degenerate. 

\subsection*{Principal eigenvalue.}
We first review the relationship between the  principal volume-normalized eigenvalue and the \emph{lattice sphere packing problem}. Recall that for a given lattice, $\Gamma_A$, the \emph{density of a sphere packing with centers at $\Gamma_A$} is given by 
$$
\mathcal{P}(A) = \textrm{proportion of space that is occupied by the spheres}.
$$
The \emph{kissing number}, $\tau(A)$, associated with the sphere packing is the number of other spheres each sphere touches. 

Using the Courant-Fischer formulae \eqref{e:CF} with $k = 1$, 
$
\lambda_1(T_B) 
= \min_{E \in \mathbb Z^d \setminus \{ 0 \}}  \ 4 \pi^2 \| B^{-t} v \|^2, 
$
we see that $\sqrt{\frac{\lambda_1}{4 \pi^2}}$ is the length of the shortest vector in the lattice $\Gamma^*_B$. The density of a packing of balls with centers on the dual lattice, $\Gamma_B^*$ is
$$
\mathcal{P}(B^\ast) = \frac{\textrm{volume of ball}}{ \textrm{volume of fundamental region}} = \frac{\omega_d \rho^d}{ |\det B^{-t}  | }
=\omega_d \rho^d | \det B | , 
$$
where $\rho$ is the radius of the balls. Observing that the shortest vector in the lattice is exactly twice the radius of the ball packing, we have 
$\sqrt{\frac{\lambda_1}{4 \pi^2}} = 2\rho$, giving 
$\rho^2 = \frac{\lambda_1}{16 \pi^2}$. 
It then follows that 
$
\mathcal{P}^{\frac 2 d} = \omega_d^{\frac 2 d} \rho^2 ( \det B)^{\frac 2 d} 
= \omega_d^{\frac 2 d} \frac{\lambda_1}{16 \pi^2} ( \det B)^{\frac 2 d}
= \omega_d^{\frac 2 d} \frac{1}{16 \pi^2} \Lambda_1
$. Rearranging gives the following lemma. 
\begin{lem} \label{p:Equiv}
Let $B \in GL(d,\mathbb R)$ and let $\Lambda_{1,d}(B) = \lambda_1(T_B) \cdot \mathrm{vol}(T_B)^{\frac 2 d}$ be the corresponding principal volume-normalized eigenvalue of the flat torus $T_B := \mathbb R^d / \Gamma_B$. 
Let $\mathcal{P}(B^\ast)$ be the packing density for the arrangement of balls with centers on the dual lattice, $\Gamma_B^* = B^{-t}\mathbb Z^d$. Then 
\begin{equation}
\Lambda_{1,d}(B) = 16 \pi^2 \omega_d^{- \frac 2 d} \mathcal{P}(B^\ast)^{\frac 2 d}, 
\end{equation}
where $\omega_d$ denotes the volume of a $d$-dimensional ball. 
Furthermore, the kissing number, $\tau(B^\ast)$, of the packing is the multiplicity of $\lambda_1(T_B)$.
\end{lem} 
A consequence of Lemma~\ref{p:Equiv} is that \emph{the eigenvalue optimization problem in  \eqref{e:OptProb} for $k=1$ is equivalent to finding the densest lattice packing of balls in $d$-dimensions.} We can also restate the eigenvalue problem in terms of the \emph{Gram matrix} of $B^{-1}_{k,d}$, denoted  
$$
G_{k,d} =(B_{k,d})^{-1} (B_{k,d})^{-t} \in \mathbb R^{d\times d}.
$$
When $k=1$, \eqref{e:OptProb} can be written 
\begin{equation} \label{Hermite}
\frac{\Lambda_{1,d}^\star}{4 \pi^2} = \max_{G\in \mathcal{S}^{d}_{>0}}
\min_{v \in \mathbb Z^d \setminus{0}}  \frac{ v^{t} G v}{ \det{G}^{\frac 1 d}}.
\end{equation}
where $\mathcal{S}^{d}_{>0}$ is the space of positive definite quadratic forms. The right hand side term of (\ref{Hermite}) 
is the so-called \emph{Hermite's constant}. It is known that finding Hermite's constant is equivalent to determining the densest lattice sphere packing \cite[6]{schurmann2009computational}.

Much is known about the densest lattice packings for small dimensions, $d$ \cite{Conway_1999}. 
In particular, this problem is NP-hard \cite{ajtai1998shortest,micciancio2001shortest}, but the densest known lattices for dimension $d=1,\ldots,8$ are known via Voronoi’s algorithm for the enumeration of perfect positive definite quadratic forms \cite{schurmann2009computational}. The  corresponding largest volume-normalized Laplacian eigenvalues are  tabulated in Table~\ref{t:Lam1}. We refer the reader to \cite{Conway_1999} for details about the lattices and to the website of
G. Nebe \cite{NebeWebsite} for explicit bases and Gram matrices for these lattices. 
Note that the multiplicity of $\lambda_1$ for these flat tori is very large. 
We also note that recently, for dimension $d=8$ and  $k=1$, the $E_8$ lattice was proven to be the maximizer of  \eqref{e:OptProb} using a different technique \cite{Viazovska_2017}.

In this paper, we focus on dimensions $d\leq 8$, but we briefly remark that this problem for the principal volume-normalized eigenvalue has been studied in higher dimensions. In particular, \cite{Conway_1999} and \cite{Marcotte_2013} give the densest known lattices for higher dimensions and the Leech lattice was proven to give the densest lattice sphere packing in dimension 24  \cite{Cohn_2017}.

\begin{table}[t]
\begin{center}
\begin{tabular}{l  l l l l l } 
$d$ & $\Gamma_B^*$  &  $\tau(B^\ast)$ & $\mathcal P(B^\ast)$ & $\Lambda_{1,d}(B)$  \\ 
\hline 
1 & $A_1$ & 2 & 1 & $4 \pi^2 \approx 39.4784 $ \\
2 & $A_2$ & 6 & $\frac{\pi}{2 \sqrt{3}} \approx 0.9069 $ \cite{de1773recherches} & $\frac{ 8\pi^2}{ \sqrt{3}} \approx 45.5858$ \\
3 & $A_3 = D_3$ & 12 & $\frac{\pi}{3 \sqrt{2}} \approx 0.7405$ \cite{gauss1840untersuchungen} & $ 4\pi^2 2^{\frac 1 3} \approx 49.7397$  \\
4 & $D_4$ & 24 & $\frac{\pi^2}{16} \approx 0.6169$ \cite{korkine1877formes} & $4\pi^2 \sqrt{2} \approx 55.8309$  \\
5 & $D_5$ & 40 & $\frac{4 \pi^2}{15} 2^{-\frac 5 2} \approx 0.4653 $ \cite{korkine1877formes} & $4 \pi^2 2^{\frac{3}{5}}  \approx 59.8381$ \\
6 & $E_6$ & 72 & $\frac{\pi^3}{48 \sqrt{3} } \approx 0.3729 $ \cite{blichfeldt1935minimum} &  $8 \pi^2 3^{-\frac 1 6}\approx 65.7460$ \\
7 & $E_7$ & 126 &$\frac{\pi^3}{105} \approx 0.2953$ \cite{blichfeldt1935minimum}  & $4 \pi^2 2^{\frac 6 7}\approx 71.5131$ \\
8 & $E_8$ & 240 & $\frac{\pi^4}{384} \approx 0.2537$ \cite{blichfeldt1935minimum}  &  $8 \pi^2 \approx 78.9568$
\end{tabular} 
\end{center}
\caption{For dimensions $d=1,\ldots, 8$, we tabulate the lattice with the largest known density, 
the corresponding kissing number $\tau(B^\ast)$, 
the density $\mathcal{P}(B^\ast)$, 
and the volume-normalized eigenvalue of the torus, $\Lambda_{1,d}(B)$. All values except $\Lambda_{1,d}$ can be obtained from \cite[Table 1.2]{Conway_1999}. }
\label{t:Lam1}
\end{table}

\subsection*{Higher eigenvalues.}
For higher values of $k$, the eigenvalue optimization problem in \eqref{e:OptProb} is less well-studied. 
Recently, Jean Legac\'e  observed that using the test lattice basis 
$ \tilde B_{k,d} = \mathrm{diag}(1, \cdots, 1, \frac{\kappa}{2}) \in \mathbb R^{d\times d}$, 
one can obtain the lower bound on the maximal value,
\begin{equation} \label{e:lowBnd}
\Lambda_{k,d}^\star \geq \Lambda_{k,d} (\tilde B_{k,d} ) = 2^{2 - \frac 2 d} \ \pi^2  \ \kappa^{\frac 2 d}, 
\qquad \qquad k,d \in \mathbb N. 
\end{equation}
Comparing \eqref{e:lowBnd} with Weyl's law \eqref{e:Weyl}, he observed that this is a meaningful bound  if $\omega_d \leq 2 = \omega_1$, which holds for $2\leq d \leq 10$. 
He further proved that, for  $2\leq d \leq 10$,  the optimal tori degenerate as $k \to \infty$ \cite{Lagac__2019}.

\subsection*{Summary of main results.} As a result of extensive computations, for dimensions $d=2,\ldots,8$ and all $k\geq 1$, we have identified $d$-dimensional flat tori 
$T^\circ_{k,d} := \mathbb R^d / \Gamma_{B_{k,d}^\circ}$, generated by 
lattices bases, $B_{k,d}^\circ$ which have very large $k$-th volume-normalized eigenvalue, $\Lambda_{k,d}^\circ := \Lambda_{k,d}(B_{k,d}^\circ)$. The bases $B_{k,d}^\circ$ have the largest objective function for the optimization problem \eqref{e:OptProb} that we were able to identify. 
Rather than report the basis matrices, we report the corresponding Gram matrices for $(B_{k,d}^\circ)^{-1}$, 
$$
G^\circ_{k,d} =(B_{k,d}^\circ)^{-1} (B_{k,d}^\circ)^{-t} \in \mathbb R^{d\times d}, 
$$
which have a nicer form. Define the $\mathbb Z^{8\times 8}$ matrix
$$
\mathcal{G}_k = \begin{pmatrix}
2 \kappa^2 & \kappa^2 &   \kappa^2 &  0 &  \kappa^2 & 0 &  \kappa^2 & -4 \\
\kappa^2&   2\kappa^2&     0&     0&     0&     0& \kappa^2& -4 \\
\kappa^2&       0& 2 \kappa^2&     0&    \kappa^2&       0&       0&  0 \\
0&       0&     0& 2 \kappa^2&       -\kappa^2&       \kappa^2&     -\kappa^2&  0 \\
\kappa^2&     0 &  \kappa^2&     -\kappa^2&   2 \kappa^2&       0&  \kappa^2&  0 \\
0&     0&     0&     \kappa^2&       0&    2\kappa^2 & - \kappa^2& 0 \\
\kappa^2& \kappa^2&     0&   -\kappa^2&  \kappa^2&  -\kappa^2&    2 \kappa^2& -4 \\
-4&      -4&     0&     0&      0&      0&      -4&  8 
\end{pmatrix}. 
$$
The Gram matrix $G^\circ_{k,d}$ is defined to be the $d\times d$ lower-right submatrix of $\mathcal{G}_k$ for each $k\geq 1$. 
A lattice basis, $B_{k,d}^\circ$ can be recovered from $G^\circ_{k,d}$ via the Cholesky decomposition. 
The nesting of the Gram matrices is a result of the dual lattices generated by the basis $(B_{k,d}^\circ)^{-t}$ being \emph{laminated}, \ie,  $(B_{k,d}^\circ)^{-t} =  \begin{pmatrix} \multirow{2}{*}{$b$} & - \ \ 0 \ \ -   \\ & (B_{k,d-1}^\circ)^{-t} \end{pmatrix}$ for some \emph{gluing vector} $b \in \mathbb R^d$. 
For example, in dimension $d=2$ we have $G^\circ_{k,2} = \begin{pmatrix} 2\kappa^2 & -4 \\ -4 & 8\end{pmatrix} \propto B_{k,2}^{-1} B_{k,2}^{-t} $, where $B_{k,2}$ is defined in \eqref{e:Bk2}. 

The following Numerical Observation\footnote{In this paper, we will use the terminology ``Numerical Observation'' to succinctly state results that depend on numerical computations. ``Theorem'' will be reserved for statements that can be proven without numerical computation.} summarizes the results of numerous computations for the flat tori $T^\circ_{k,d}$,  and their volume-normalized eigenvalues, $\Lambda_{k,d}^\circ$. 

\begin{table}[t]
\begin{center}
\renewcommand{\arraystretch}{1.2}
{\footnotesize
\begin{tabular}{|c|c|c|c|c|c|c|c|c|}
\hline 
$d$ & 1 & 2 & 3 & 4 & 5 & 6 & 7 & 8 \\
\hline 
$ \frac{\Lambda_{k,d}^{\circ} }{ h_d  \pi^2}$ 
& $ \kappa^{2}$ 
& $\left( \frac{  \kappa^{4}}{\kappa^{2}-1} \right)^{\frac 1 2}$ 
& $\left(\frac{\kappa^{4}}{\kappa^{2}-\frac{4}{3}}\right)^{\frac{1}{3}}$
& $\left( \frac{\kappa^4 }{\kappa^{2}-2}\right)^{\frac{1}{4}}$ 
& $ \left(\frac{\kappa^{4}}{\kappa^{2}-2}\right)^{\frac{1}{5}}$
& $ \left(\frac{\kappa^{4}}{\kappa^{2}- \frac{5}{2}}\right)^{\frac{1}{6}}$ 
& $  \left(\frac{\kappa^{4}}{\kappa^{2}-\frac{8}{3}}\right)^{\frac{1}{7}}$ 
& $ \left(\frac{\kappa^{4}}{\kappa^{2}-3}\right)^{\frac{1}{8}}$ \\
\hline 
$h_d$ &1 &2
&$4\cdot 3^{-\frac{1}{3}}$
&$2^{\frac 7 4}$ 
&4 
&$2^{\frac{11}{6}}$
&$4 \left( \frac{16}{3}\right)^{\frac 1 7} $
&$2^{\frac 5 2}$\\
\hline
$h_d/g_d$ 
&1
&$\frac{\pi}{2} $
&$\frac{2}{3} 2^{\frac 1 3} \pi^{\frac 2 3} $
&$\pi2^{- \frac 3 4} $
&$\frac{2 \sqrt[5]{2} \pi ^{4/5}}{15^{2/5}} $
&$\frac{\pi }{\sqrt{2} \sqrt[3]{3}} $
& $\frac{2\ 2^{5/7} \pi ^{6/7}}{3^{3/7} 35^{2/7}} $
&$\pi 6^{- \frac 1 4} $ \\
\hline
$h_d/g_d \approx$ 
&1
&$1.57$
&$ 1.80$
&$ 1.87$
&$ 1.94$
&$ 1.54$
&$ 1.98$
&$ 2.01$ \\
\hline
$k=1$ mult. & 2 & 6 & 12 & 24 & 40 & 72 & 126 & 240 \\
\hline 
$k\geq 2$ mult. & 2 & 6 & 12 & 22 & 38 & 62 & 106 & 182 \\
\hline
$| \det G^\circ_{k,d}|/8$ & 
 1 & 
 $2( \kappa^2 -1)$  &  
 $\kappa^2 (3\kappa^2 - 4)$ & 
 $4 \kappa^4(\kappa^2-2)$ & 
 $4 \kappa^6 (\kappa^2 - 2)$ & 
 $2 \kappa^8 (2 \kappa^2 - 5)$ & 
 $\kappa^{10} (3\kappa^2 - 8)$ & 
 $2 \kappa^{12} (\kappa^2 - 3)$ \\
\hline
\end{tabular}}
\caption{
$\frac{\Lambda_{k,d}^\circ } {h_d \pi^2}$, $h_d$
$h_d/g_d$, 
eigenvalue multiplicities for $k=1$ and $k\geq2$. 
See Numerical Observation~\ref{p:OptLattice} and following discussion in Section~\ref{s:Intro}.} 
\label{t:EigTable}
\end{center}
\end{table}

\begin{no} \label{p:OptLattice}
For $k\geq 1$ and $d \leq 8$,  
the flat tori $T^\circ_{k,d}$ have $k$-th volume-normalized eigenvalues 
$
\Lambda_{k,d}^\circ := \Lambda_{k,d}( B_{k,d}^\circ)
$  as tabulated in the second row of Table~\ref{t:EigTable}. 
The multiplicity of the eigenvalues are given in the sixth and seventh rows of Table~\ref{t:EigTable}.
The corresponding lattice vectors are of the form $\pm v$ where $v$ is a vector tabulated in Table~\ref{t:LattVecs}.
\end{no}

Details on our computations supporting Numerical Observation~\ref{p:OptLattice} are given in Section~\ref{s:proof:OptLattice}. Magma code with these supporting computations can be found at  \cite{github}. 

We plot $k$ vs $\Lambda_{k,d}^\circ$ for  $d\leq 8$ in Figure~\ref{f:Lamkd} and tabulate the first few values in Table~\ref{t:Lamkd}. 
Using the observation that for $a>0$, $\frac{\kappa^4}{\kappa^2-a} \geq k^2$, we have that 
for each dimension $d\leq 8$, 
$$ 
\Lambda_{k,d}^\circ \geq  h_d \ \pi^2 \ \kappa^\frac{2}{d}, 
\qquad \qquad 
\forall k \geq 1, 
$$
where  $\kappa :=2 \left\lceil \frac{k}{2} \right\rceil$ and $h_d$  is a constant,  which does not depend on $k$, as tabulated in the third row of Table~\ref{t:EigTable}. In particular, this shows that the optimal value in \eqref{e:OptProb} satisfies 
 $$
\Lambda_{k,d}^\star \geq  h_d \  \pi^2 \ \kappa^\frac{2}{d},
\qquad \qquad 
\forall k \geq 1.
$$
In the fourth row of Table~\ref{t:EigTable}, we compute the value of $h_d/g_d$, where $g_d$ is the constant appearing in  Weyl's law. 
\emph{Depending on the dimension, the Laplace eigenvalues of $T^\circ_{k,d}$ exceed the value in Weyl's law by a factor between 1.54 and 2.01} as indicated in the fifth row of Table~\ref{t:EigTable}.

The eigenvalue multiplicities listed in Table~\ref{t:EigTable} are very large. This is a consequence of the fact that all lattice vectors $v \in \mathbb Z^8$ in Table~\ref{t:LattVecs}  satisfy 
$v^t G^\circ_{k,d} v = 8 \left\lceil \frac{k}{2} \right\rceil^2=2 \kappa^2$. 
Note that the first vector in the table is 
$\left(0,\cdots,0,\left\lceil \frac{k}{2} \right\rceil \right) \in \mathbb Z^8$. The $k-1$ non-trivial lattice vectors, $v$ with smaller value of $v\mapsto v^t G^\circ_{k,d} v$ are of the form $\pm \left(0,\cdots,0,\left\lceil \frac{j}{2} \right\rceil \right)$, where $j=1,\ldots, k-1$. 

In Section~\ref{s:eigPert}, we give a condition for stationarity; see Theorem~\ref{t:NeccCond}. 
In Section~\ref{s:LocOpt}, we show numerically that the bases $B_{k,d}^\circ$ satisfy this stationarity condition for $k\geq 1$ and $d\leq 8$. In Section~\ref{s:Degeneracy}, we also  show numerically that the tori $T^\circ_{k,d}$ degenerate as $k \to \infty$. Supporting Sage code for these numerical claims is available at \cite{github}. 

Finally, in Section~\ref{s:NumMeth}, we describe the numerical methods that were used to compute the locally maximal solutions to the optimization problem in  \eqref{e:OptProb} and compute the bases $B^\circ_{k,d}$, described above. Briefly, the optimization problem in \eqref{e:OptProb} was solved by solving a sequence of linearized problems, similar to the method in \cite{Marcotte_2013} for the closest packing problem ($k=1$).  We have also used these methods to investigate \eqref{e:OptProb} for $d>8$. Although we have identified locally optimal solutions in higher dimensions, we were not able to identify laminated structure in these higher dimensional lattices.

\subsection*{Other related work}
We briefly mention that Milnor used the relationship between flat tori and lattices to find two 16 dimensional compact Riemannian Manifolds that have the same Laplace spectrum (isospectral) but are not isometric \cite{Milnor_1964}. 

In this paper, we count the length of lattice vectors \emph{with multiplicity}. However, we could consider the problem where we enumerate the length of vectors in $\Gamma_B$ in increasing order \emph{without multiplicity}, 
$$
0 = \nu_0 < \nu_1 < \nu_2 < \cdots,
$$
where $\nu_k$ is called the $k$-th  length of $\Gamma_B$. In this setting, for dimensions 2 to 8, Paul Schmutz Schaller \cite{schaller1998geometry} conjectured that the lattices with best known sphere packings have maximal lengths, \ie, for all $k>0$ their $k$-th  length is strictly greater than the $k$-th  length of any other lattice in the same dimension with the same covolume. This problem is also equivalent to the extremal $k$-th length of closed geodesics among the flat tori of the same dimension and volume.
In \cite{willging2008conjecture}, Willging showed that the conjecture is false in dimension 3 and demonstrated that the $6$-th shortest vector of the honeycomb lattice is longer than the $6$-th shortest vector of  the face-centered cubic lattice, which is the optimal lattice for sphere packing in dimension 3.

\begin{figure}[t!]
\centering
\includegraphics[width=0.8\textwidth,trim={0 0 1cm 1cm},clip]{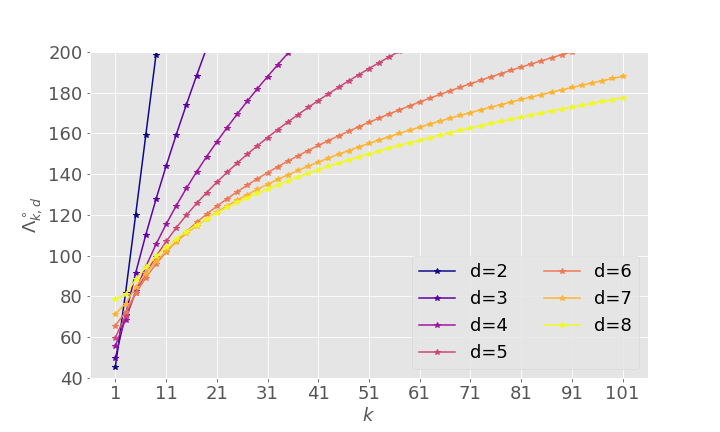}
\caption{\label{f:Lamkd} For indicated dimensions $d=1,\ldots,8$, a plot of $k$ vs. $\Lambda^\circ_{k,d}$. }
\end{figure}

\begin{table}[t!]
\begin{center}
\vspace{1cm}
\begin{tabular}{lrrrrrrrr}
\hline
 $k \setminus d$   &         1 &         2 &         3 &         4 &         5 &         6 &         7 &         8 \\
\hline
 1,2               &    39.478 &    45.586 &    49.740 &    55.831 &    59.838 &    65.746 &    71.513 &    78.957 \\
 3,4               &   157.914 &    81.546 &    71.005 &    68.648 &    70.596 &    72.363 &    76.480 &    81.033 \\
 5,6               &   355.306 &   120.115 &    91.527 &    82.487 &    81.768 &    81.494 &    84.590 &    88.336 \\
 7,8               &   631.655 &   159.162 &   110.262 &    94.644 &    91.275 &    89.217 &    91.387 &    94.461 \\
 9,10              &   986.960 &   198.387 &   127.623 &   105.511 &    99.567 &    95.873 &    97.187 &    99.662 \\
 11,12             &  1421.223 &   237.697 &   143.920 &   115.401 &   106.966 &   101.748 &   102.262 &   104.187 \\
 13,14             &  1934.442 &   277.057 &   159.365 &   124.532 &   113.685 &   107.029 &   106.790 &   108.204 \\
 15,16             &  2526.619 &   316.446 &   174.109 &   133.050 &   119.864 &   111.845 &   110.892 &   111.826 \\
 17,18             &  3197.752 &   355.855 &   188.263 &   141.062 &   125.605 &   116.283 &   114.651 &   115.132 \\
 19,20             &  3947.842 &   395.279 &   201.909 &   148.649 &   130.981 &   120.410 &   118.128 &   118.179 \\
\hline
\end{tabular}

\end{center}
\caption{ $\Lambda^\circ_{k,d}$ for the indicated values of the eigenvalue number $k$ and dimension $d$. }
\vspace{1cm}
\label{t:Lamkd}
\end{table}

\begin{table}[ht!]
\vspace{-.3cm}
\scriptsize
\renewcommand{\arraystretch}{0.95}
\setlength\columnsep{0pt}
\begin{multicols}{2}
\begin{tabular}{r|r|cccccccc}    
$k=1,2$ & $k\geq3$ & \multicolumn{8}{l}{lattice vector} \\
\hline 
1 & \ko{1}& 0 &  0 &  0 &  0 &  0 &  0 &  0 &  $\lceil k/2 \rceil$ \\   
\hline
2 & 2& 0 &  0 &  0 &  0 &  0 &  0 &  1 &  0 \\   
3 & 3& 0 &  0 &  0 &  0 &  0 &  0 &  1 &  1 \\   
\hline
4 & \ko{4}& 0 &  0 &  0 &  0 &  0 &  1 &  0 &  0 \\   
5 & 5& 0 &  0 &  0 &  0 &  0 &  1 &  1 &  0 \\   
6 & 6& 0 &  0 &  0 &  0 &  0 &  1 &  1 &  1 \\   
\hline
7 & \ko{7}& 0 &  0 &  0 &  0 &  1 &  0 &  0 &  0 \\   
8& 8 & 0 &  0 &  0 &  0 &  1 &  0 & -1 &  0 \\   
9 & 9& 0 &  0 &  0 &  0 &  1 & -1 & -1 &  0 \\  
10 & 10& 0 &  0 &  0 &  0 &  1 &  0 & -1 & -1 \\  
11 & 11& 0 &  0 &  0 &  0 &  1 & -1 & -1 & -1 \\  
12 & & 0 &  0 &  0 &  0 &  1 & -1 & -2 & -1 \\  
\hline
13 & \ko{12} & 0 &  0 &  0 &  1 &  0 &  0 &  0 &  0 \\  
14 & \ko{13} & 0 &  0 &  0 &  1 &  0 & -1 &  0 &  0 \\  
15 & \ko{14} & 0 &  0 &  0 &  1 &  1 & -1 &  0 &  0 \\  
16 & \ko{15} & 0 &  0 &  0 &  1 &  1 &  0 &  0 &  0 \\  
17 & 16 & 0 &  0 &  0 &  1 &  0 &  0 &  1 &  1 \\  
18 & 17 & 0 &  0 &  0 &  1 &  0 &  0 &  1 &  0 \\  
19 & 18 & 0 &  0 &  0 &  1 &  1 & -1 & -1 & -1 \\  
20 & 19 & 0 &  0 &  0 &  1 &  1 & -1 & -1 &  0 \\  
\hline
21 & \ko{20} & 0 &  0 &  1 &  0 &  0 &  0 &  0 &  0 \\  
22 & \ko{21} & 0 &  0 &  1 & -1 & -1 &  1 &  0 &  0 \\  
23 & \ko{22} & 0 &  0 &  1 &  0 & -1 &  0 &  0 &  0 \\  
34 & \ko{23} & 0 &  0 &  1 & -1 & -1 &  0 &  0 &  0 \\  
25 & 24 & 0 &  0 &  1 &  0 & -1 &  0 &  1 &  1 \\  
26 & 25 & 0 &  0 &  1 &  0 & -1 &  0 &  1 &  0 \\  
27 & 26 & 0 &  0 &  1 &  0 & -1 &  1 &  1 &  1 \\  
28 & 27 & 0 &  0 &  1 & -1 & -1 &  1 &  1 &  1 \\  
29 & 28 & 0 &  0 &  1 &  0 & -1 &  1 &  1 &  0 \\  
30 & 29 & 0 &  0 &  1 & -1 & -1 &  1 &  1 &  0 \\  
31 & 30 & 0 &  0 &  1 & -1 & -2 &  1 &  1 &  1 \\  
32 & 31 & 0 &  0 &  1 & -1 & -2 &  1 &  1 &  0 \\  
33 & & 0 &  0 &  1 & -1 & -2 &  2 &  2 &  1 \\  
34 & & 0 &  0 &  1 &  0 & -1 &  1 &  2 &  1 \\  
35 & & 0 &  0 &  1 &  0 & -2 &  1 &  2 &  1 \\  
36 & & 0 &  0 &  1 & -1 & -2 &  1 &  2 &  1 \\  
\hline
37 & 32 & 0 &  1 &  0 &  0 &  0 &  0 &  0 &  0 \\  
38 & \ko{33} & 0 &  1 &  0 &  0 &  0 & -1 & -1 &  0 \\  
39 & \ko{34} & 0 &  1 &  0 &  0 &  0 &  0 & -1 &  0 \\  
40 & \ko{35} & 0 &  1 &  0 & -1 &  0 &  0 & -1 &  0 \\  
41 & \ko{36} & 0 &  1 &  0 &  0 &  1 &  0 & -1 &  0 \\  
42 & \ko{37} & 0 &  1 & -1 &  0 &  1 &  0 & -1 &  0 \\  
43 & \ko{38} & 0 &  1 &  0 &  0 &  1 & -1 & -1 &  0 \\  
44 & \ko{39} & 0 &  1 & -1 &  0 &  1 & -1 & -1 &  0 \\  
45 & \ko{40} & 0 &  1 &  0 &  1 &  1 & -1 & -1 &  0 \\  
46 & \ko{41} & 0 &  1 & -1 &  1 &  1 & -1 & -1 &  0 \\  
47 & \ko{42} & 0 &  1 & -1 &  1 &  2 & -1 & -1 &  0 \\  
48 & 43 & 0 &  1 &  0 &  0 &  1 & -1 & -2 &  0 \\  
49 & 44 & 0 &  1 & -1 &  0 &  1 & -1 & -2 &  0 \\  
50 & 45 & 0 &  1 &  0 &  0 &  1 & -1 & -2 & -1 \\  
51 & 46 & 0 &  1 & -1 &  0 &  1 & -1 & -2 & -1 \\  
52 & 47 & 0 &  1 & -1 &  1 &  2 & -2 & -2 &  0 \\  
53 & 48 & 0 &  1 & -1 &  1 &  2 & -2 & -2 & -1 \\  
54 & 49 & 0 &  1 &  0 &  0 &  0 &  0 &  0 &  1 \\  
55 & 50 & 0 &  1 & -1 &  0 &  2 & -1 & -2 &  0 \\  
56 & 51 & 0 &  1 & -1 &  1 &  2 & -1 & -2 &  0 \\  
57 & 52 & 0 &  1 & -1 &  0 &  2 & -1 & -2 & -1 \\  
58 & 53 & 0 &  1 & -1 &  1 &  2 & -1 & -2 & -1 \\  
59 & & 0 &  1 & -2 &  1 &  3 & -2 & -3 & -1 \\
60 & & 0 &  1 & -1 &  1 &  2 & -2 & -3 & -1 
\end{tabular}

\begin{tabular}{r|r|cccccccc}
$k=1,2$ & $k\geq3$ & \multicolumn{8}{l}{lattice vector} \\
\hline 
61 & & 0 &  1 & -1 &  0 &  2 & -2 & -3 & -1 \\  
62 & & 0 &  1 & -1 &  0 &  2 & -1 & -3 & -1 \\  
63 & & 0 &  1 & -1 &  1 &  3 & -2 & -3 & -1 \\  
\hline
64 & 54 & 1 &  0 &  0 &  0 &  0 &  0 &  0 &  0 \\  
65 & \ko{55} & 1 &  0 &  0 & -1 & -1 &  0 & -1 &  0 \\  
66 & \ko{56} & 1 &  0 &  0 &  0 &  0 & -1 & -1 &  0 \\  
67 & \ko{57} & 1 &  1 & -1 &  0 &  1 & -1 & -2 &  0 \\  
68 & \ko{58} & 1 &  0 &  0 &  0 &  0 &  0 & -1 &  0 \\  
69 & \ko{59} & 1 &  0 &  0 & -1 &  0 &  0 & -1 &  0 \\  
70 & \ko{60} & 1 &  0 & -1 &  0 &  0 & -1 & -1 &  0 \\  
71 & \ko{61} & 1 & -1 &  0 & -1 & -1 &  1 &  0 &  0 \\  
72 & \ko{62} & 1 & -1 &  0 &  0 & -1 &  0 &  0 &  0 \\  
73 & \ko{63} & 1 & -1 &  0 & -1 & -1 &  0 &  0 &  0 \\  
74 & \ko{64} & 1 &  0 & -1 &  0 &  0 &  0 & -1 &  0 \\  
75 & \ko{65} & 1 &  0 & -1 & -1 &  0 &  0 & -1 &  0 \\  
76 & \ko{66} & 1 & -1 & -1 &  0 &  0 &  0 &  0 &  0 \\  
77 & \ko{67} & 1 & -1 &  0 &  0 &  0 &  0 &  0 &  0 \\  
78 & \ko{68} & 1 &  0 & -1 &  0 &  1 &  0 & -1 &  0 \\  
79 & \ko{69} & 1 &  0 & -1 &  0 &  1 & -1 & -1 &  0 \\  
80 & \ko{70} & 1 &  0 & -1 &  1 &  1 & -1 & -1 &  0 \\ 
81 & 71 & 1 & -1 &  0 & -1 & -2 &  1 &  1 &  0 \\  
82 & 72 & 1 & -1 &  0 &  0 & -1 &  0 &  1 &  0 \\  
83 & 73 & 1 & -1 &  0 &  0 & -1 &  0 &  1 &  1 \\  
84 & 74 & 1 &  0 & -1 &  0 &  1 & -1 & -2 & -1 \\  
85 & 75 & 1 &  0 & -1 &  0 &  1 & -1 & -2 &  0 \\  
86 & 76 & 1 & -1 &  0 &  0 & -1 &  1 &  1 &  0 \\  
87 & 77 & 1 & -1 &  0 & -1 & -1 &  1 &  1 &  0 \\  
88 & 78 & 1 & -1 &  0 &  0 & -1 &  1 &  1 &  1 \\  
89 & 79 & 1 & -1 &  0 & -1 & -1 &  1 &  1 &  1 \\  
90 & 80 & 1 &  0 &  0 & -1 & -1 &  1 &  0 &  0 \\  
91 & 81 & 1 & -1 &  0 & -1 & -2 &  1 &  1 &  1 \\  
92 & 82 & 1 &  0 &  0 &  0 & -1 &  0 &  0 &  0 \\  
93 & 83 & 1 &  0 &  0 & -1 & -1 &  0 &  0 &  0 \\  
94 & 84 & 1 &  0 &  0 &  0 &  0 &  0 &  0 &  1 \\  
95 & 85 & 1 &  0 & -1 &  0 &  0 &  0 &  0 &  0 \\  
96 & 86 & 1 &  0 &  0 & -1 & -1 &  1 &  0 &  1 \\  
97 & 87 & 1 &  0 &  0 &  0 & -1 &  0 &  0 &  1 \\  
98 & 88 & 1 &  0 &  0 & -1 & -1 &  0 &  0 &  1 \\  
99 & 89 & 1 &  0 & -1 &  0 &  0 &  0 &  0 &  1 \\  
100 & 90 & 1 & -1 &  1 & -1 & -2 &  1 &  1 &  0 \\ 
101 & 91 & 1 & -1 &  1 & -1 & -2 &  1 &  1 &  1 \\ 
102 & & 1 & -1 &  0 &  0 & -1 &  1 &  2 &  1 \\ 
103 & & 1 & -1 &  1 & -1 & -3 &  2 &  2 &  1 \\ 
104 & & 1 & -1 &  1 & -2 & -3 &  2 &  2 &  1 \\
105 & & 1 & -1 &  0 &  0 & -2 &  1 &  2 &  1 \\ 
106 & & 1 &  0 &  1 & -1 & -2 &  1 &  1 &  1 \\ 
107 & & 1 & -1 &  0 & -1 & -2 &  2 &  2 &  1 \\ 
108 & & 1 &  0 &  0 &  0 & -1 &  0 &  1 &  1 \\ 
109 & & 1 & -1 &  1 &  0 & -2 &  1 &  2 &  1 \\ 
110 & & 1 &  0 &  0 & -1 & -2 &  1 &  1 &  1 \\ 
111 & & 1 &  0 &  0 & -1 & -1 &  1 &  1 &  1 \\
112 & & 1 & -1 &  1 & -1 & -3 &  2 &  3 &  1 \\ 
113 & & 1 & -2 &  1 & -1 & -3 &  2 &  3 &  1 \\ 
114 & & 1 & -1 &  1 & -1 & -3 &  2 &  3 &  2 \\ 
115 & & 1 & -1 &  1 & -1 & -2 &  2 &  2 &  1 \\ 
116 & & 1 & -1 &  1 & -1 & -3 &  1 &  2 &  1 \\ 
117 & & 1 &  0 &  0 &  0 & -1 &  1 &  1 &  1 \\ 
118 & & 1 & -1 &  0 & -1 & -2 &  1 &  2 &  1 \\ 
119 & & 1 & -1 &  1 & -1 & -2 &  1 &  2 &  1 \\ 
120 & & 2 & -1 &  0 & -1 & -2 &  1 &  1 &  1
\end{tabular}
\end{multicols}
\vspace{-.25cm}
\caption{For the $B^\circ_{k,d}$-lattice, $k$-th shortest lattice vectors and indexing for $k=1,2$ and $k\geq 3$. For $d<8$, as indicated by horizontal lines, we use only vectors that are zero in the first $8-d$ components. The red indices are discussed in Section \ref{s:LocOpt}.}
\label{t:LattVecs}
\end{table}

 \clearpage

\section{Comments on the computations supporting Numerical Observation~\ref{p:OptLattice}} \label{s:proof:OptLattice}

Here, we discuss the claims in Numerical Observation~\ref{p:OptLattice} 
regarding the
$k$-th volume-normalized Laplacian eigenvalues of the torus $T^\circ_{k,d}$, 
\begin{equation}
\Lambda^\circ_{k,d} :=   \min_{E \in \mathbb Z^d_{k+1}} \ \max_{v \in  E} \  4 \pi^2 (\det B^\circ_{k,d} )^{\frac 2 d} \| (B^\circ_{k,d})^{-t} v \|^2. 
\qquad 
\label{e:contour}
\end{equation}
and the corresponding lattice vectors, $E$.
For $k=1$, the computation of $\Lambda^\circ_{k,d}$ is known as the shortest lattice vector problem (SVP) for the dual lattice, $\Gamma^*_{B^\circ_{k,d}} = (B^\circ_{k,d})^{-t} \mathbb Z^d$. 
The SVP appears in a variety of cryptoanalysis problems and, although NP-hard \cite{ajtai1998shortest,micciancio2001shortest}, 
can be solved \emph{for fixed k} in moderately-high dimensions \cite{Hanrot_2011,Mariano_2017}.
\emph{We are unaware of a method to find the shortest $k$ vectors of the lattice $B^\circ_{k,d}$ analytically. }

For fixed (small to moderately large) $k \in \mathbb N$ , we can compute $\Lambda^\circ_{k,d}$ in a rigorous way using the 
\verb+ShortVectors+ function\footnote{\url{http://magma.maths.usyd.edu.au/magma/handbook/text/331}} in Magma \cite{MR1484478}. 
The enumeration routine underlying this function relies on floating-point approximation, but is run in a rigorous way by using the default setting with the parameter \verb+Proof+ set to \verb+true+. For each $d = 1,\ldots, 8$, we checked all values of $k$ from 1 to 500,000 and every value 
$
k\in \{
1\times 10^6, \ 
2\times 10^6, \ 
\cdots, \ 
9\times 10^6, \ 
1 \times 10^7, \
2 \times 10^7, \
\cdots, \ 
9\times 10^7, \ 
1 \times 10^8 
\}$. 
Magma code with these supporting computations can be found in the \verb+solve_SVP.magma+ file at  \cite{github}. Indeed, the claims made in Numerical Observation~\ref{p:OptLattice} hold for these values of $k$.

\section{Eigenvalue perturbation formulae and conditions for  stationarity} 
\label{s:eigPert}

Recall that the eigenvalues of $-\Delta$ on a flat torus each have multiplicity of at least two. We will refer to an eigenvalue as a \emph{double eigenvalue} if it has multiplicity of exactly two. We first give the perturbation formula for a double eigenvalue. 

\begin{thm} \label{t:EigPertSimp}
When $\lambda$ is a double eigenvalue with corresponding lattice vectors $\pm v \in \mathbb Z^d$, the variation of the normalized eigenvalue 
$\Lambda$ with respect to the Gram matrix $G$ satisfies
\[
\Lambda\left(G_{0}+\delta G\right)=\Lambda\left(G_{0}\right)+\left\langle \frac{\partial\Lambda}{\partial G},\delta G\right\rangle_F +o\left(\left\Vert \delta G\right\Vert \right) 
\]
where
$\frac{\partial\Lambda}{\partial G} = 
-\frac{\Lambda}{d} G^{-1} + 4\pi^{2}\left(\det(G)\right)^{-\frac{1}{d}} v v^t$. 
\end{thm}
\begin{proof} 
For an invertible, symmetric matrix $G$, Jacobi's formula
states that 
\[
\det (G_0 + \delta G) = \det G_0 + 
\det G_0 \langle G^{-1}, \delta G \rangle 
+o\left(\left\Vert \delta G\right\Vert \right). 
\]
Since
\[
\Lambda=4\pi^{2}\left(\det(G)\right)^{-\frac{1}{d}}\left\langle v v^t,G  \right\rangle_F ,
\]
for fixed lattice vector $v$, 
we obtain the desired result using the product rule. 
\end{proof}

We next give a perturbation formula for eigenvalues of greater multiplicity. 
\begin{thm} \label{t:EigPertMult}
Suppose the Laplacian eigenvalue $\lambda$ has even multiplicity $m > 2$ with corresponding lattice vectors given by 
$\pm v_{j} \in \mathbb Z^d$, $j=1,\ldots, \frac{m}{2}$. 
A perturbation of the Gram matrix of the form $G = G_0 + \delta G$ will split the normalized eigenvalue  $\Lambda =  |\det G|^{- \frac 1 d} \lambda$ into up to $\frac{m}{2}$ (un-sorted) normalized eigenvalues (each with multiplicity of at least two) given by 
\[
\Lambda_{j}\left(G_{0}+\delta G\right)=\Lambda\left(G_{0}\right)+\mu_{j} 
+ o\left( \left\Vert \delta G\right\Vert \right), \qquad \qquad j=1,\ldots, \frac{m}{2}
\]
where $\mu_{j}= \langle M_j, \delta G \rangle$ and
\begin{equation}
\label{e:M}
M_j = -\frac{\Lambda}{d} G_0^{-1} + 4\pi^{2}\left( \det G_0 \right)^{-\frac{1}{d}} v_j v_j^t .
\end{equation}
\end{thm}
\begin{proof} 
The volume-normalized eigenvalue, $\Lambda$, satisfies 
\[
4\pi^{2}\left(\det(G_0)\right)^{-\frac{1}{d}}\left\langle v_{j},G_0 v_{j}\right\rangle =\Lambda, 
\qquad \qquad 
1\le j\le \frac{m}{2}.
\]
Noting that the lattice vectors $v_j$ are fixed, 
the perturbed volume normalized eigenvalues satisfy 
\[
4\pi^{2}\left(\det(G_0+\delta G)\right)^{-\frac{1}{d}}\left\langle v_{j},(G_0+\delta G) v_{j}\right\rangle = \Lambda+ \mu_j + o(\| \delta G\|), 
\qquad \qquad 
1\le j\le \frac{m}{2}.
\]
The first order terms give 
\[
-\frac{4\pi^{2}}{d}\left(\det(G_0)\right)^{-\frac{1}{d}}\, \langle G_0^{-1}, \delta G \rangle 
\left\langle v_{j},G_0 v_{j}\right\rangle +4\pi^{2}\left(\det(G_0)\right)^{-\frac{1}{d}}\left\langle v_{j},\delta G v_{j}\right\rangle =\mu_{j}, 
\qquad \qquad 
1\le j\le \frac{m}{2}.
\]
Thus
\begin{align*}
\mu_{j} & =4\pi^{2}\left(\det(G_0)\right)^{-\frac{1}{d}}\left\{ \left\langle v_{j},\delta Gv_{j}\right\rangle -\frac{1}{d} 
\langle G_0^{-1}, \delta G \rangle 
\left\langle v_{j},G_0 v_{j}\right\rangle \right\},
\end{align*}
as desired.
\end{proof}

Recall that the volume normalized eigenvalue is scale invariant, \ie, $\Lambda(\alpha G_0) = \Lambda(G_0)$ for $\alpha \neq 0$. 
In Theorem~\ref{t:EigPertMult}, if we take $\delta G = \varepsilon G_0$ we obtain $\mu_j =0$ for all $j=1,\ldots, \frac{m}{2}$ and
$\Lambda_j(G_0 + \delta G) = \Lambda(G_0) + o(\|G_0\|)$, as we expect.

We next use Theorem~\ref{t:EigPertMult} to derive two necessary conditions for local optimality in the eigenvalue optimization problem \eqref{e:OptProb}. We say that $G_0$ is a \emph{stationary point} for $\Lambda$ if for every $\delta G$ we have that there exists at least one $j \in \{ 1,\ldots,\frac{m}{2} \}$ such that 
$\mu_{j} \leq 0$. 
We say that $G_0$ is a \emph{strict local maximum} for $\Lambda$ if for every $\delta G$ satisfying $\langle \delta G, G_0 \rangle = 0$ we have at least one $j \in \{ 1,\ldots,\frac{m}{2} \}$ such that $ \mu_{j} < 0$.

\begin{thm} \label{t:linInd}
Using the notation in Theorem~\ref{t:EigPertMult}, a necessary condition for  $G_0$ to be a strict local maximum for $\Lambda$ is that the collection of outer products $\{v_j v_j^t\}_{j =1}^{\frac{m}{2}}$ spans the space of symmetric $\mathbb R^{d \times d}$ matrices. 
\end{thm}
\begin{proof}
Otherwise, there exists a matrix $\delta G = A$ so that for all $j =1,\ldots, \frac{m}{2}$, we have 
$$
\langle A, v_j v_j^t \rangle = 0. 
$$
In this case, for every $j =1,\ldots, \frac{m}{2}$, we have
$$
\mu_j = \langle M_j, A \rangle 
= - \frac{\Lambda}{d} \langle G_0^{-1}, A \rangle. 
$$
Changing the sign of $A$ if necessary, we may assume that $\langle G_0^{-1}, A \rangle \leq 0 $, implying $\mu_j \geq 0$ for all  $j =1,\ldots, \frac{m}{2}$.
\end{proof}

\begin{thm} \label{t:NeccCond}
Using the notation in Theorem~\ref{t:EigPertMult}, 
the Gram matrix $G_0 $ is a stationary point for $\Lambda$ if and only if 
there are non-negative coefficients $c_j\geq0$, $j=1,\ldots,\frac{m}{2}$,  not all zero, such that 
$\sum_{j=1}^{\frac{m}{2}} c_j M_j = 0$.
\end{thm}

\begin{proof}
We consider the linear map $U \colon \mathbb S^d_{++} \to \mathbb R^{\frac{m}{2}}$ defined by 
$$
U_j(G) := \mu_j(G) = \langle M_j, G \rangle_F,
\qquad \qquad j = 1, \ldots, \frac{m}{2}. 
$$
To find the adjoint map of $U$, denoted $U^* \colon \mathbb R^{\frac{m}{2}} \to \mathbb S^d_{++}$, for 
$c \in \mathbb R^{\frac{m}{2}}$, 
we compute 
$$
\langle U_j(G), c \rangle_{\mathbb R^{\frac{m}{2}}} 
= \sum_{j=1}^{\frac{m}{2}} c_j \langle M_j, G\rangle_F = \left\langle  \left( \sum_{j=1}^{\frac{m}{2}} c_j M_j \right), G \right\rangle_F, 
$$ 
so that 
$U^* c = \sum_{j=1}^{\frac{m}{2}} c_j M_j$.

Stationarity of $G_0$ means that
 $\delta G \mapsto U(\delta G) \in \mathbb R^{\frac m 2}$ has at least one non-positive component for every $\delta G$, \ie, there is no solution to the linear system 
\begin{equation} 
\label{e:Gordan1}
U(\delta G) > 0.
\end{equation}

We recall Gordan's Alternative Theorem (see, \eg, \cite[Thm. 10.4]{Beck_2014}) which states that either \eqref{e:Gordan1} has a solution or  
\begin{equation} 
\label{e:Gordan2}
U^* c = 0, 
\quad 
c \geq 0, \ c \neq 0
\end{equation}
has a solution. 
Thus it is enough to show that there is a non-trivial, non-negative $c \in \textrm{ker}(U^*)$.
\end{proof}

\begin{rem}
For  $k=1$, the eigenvalue optimization problem \eqref{e:OptProb} is equivalent to
Hermite's constant \eqref{Hermite} and 
determining the densest lattice sphere packing. It was shown by Voronoi that a lattice gives the densest lattice sphere packing if and only if it is perfect and eutactic \cite[Thm. 3.9]{schurmann2009computational}. It can be seen that the necessary conditions here imply  Theorems~\ref{t:linInd} and \ref{t:NeccCond}  for  $k=1$. Note that the lack of convexity for higher eigenvalues makes a sufficiency condition more difficult to state.
\end{rem}

\section{Properties of flat tori, \texorpdfstring{$T^\circ_{k,d}$}{T}, and degeneracy as \texorpdfstring{$k\to \infty$}{k tends to infinity}} \label{s:CompRes}

In this section, we show that 
$G^\circ_{k,d}$ for $d\leq 8$ and $k\geq 1$ satisfies the necessary condition for strict local maximum given in Theorem~\ref{t:linInd} (see Section~\ref{s:LinInd}) and
provide numerical evidence that it satisfies the necessary condition for stationarity in 
Theorem~\ref{t:NeccCond} (see Section~\ref{s:LocOpt}). In Section~\ref{s:Degeneracy}, we describe the degeneracy of flat tori $T^\circ_{k,d}$ as $k\to \infty$.

\subsection{Linear Independence} \label{s:LinInd}
\begin{thm}
Let $\{v_j \}_{j = 1}^{\frac m 2}$ be the collection of lattice vectors gives in Table~\ref{t:LattVecs}. 
The collection of outer products $\{v_j v_j^t\}_{j =1}^{\frac{m}{2}}$ spans the space of symmetric $\mathbb R^{d \times d}$ matrices. 
Consequently, $G^\circ_{k,d}$ for $d\leq 8$ and $k\geq 1$ satisfy the necessary conditions for a strict local maximum given in Theorem~\ref{t:linInd}. 
\end{thm}
\begin{proof}
We only need $\binom{d}{2}\leq \frac{m}{2}$ outer products $v_j v_j^t$ to span the space of symmetric matrices, so for dimensions, $d = 4,5,6,7,8$, it is not necessary to use all of the lattice vectors listed in Table~\ref{t:LattVecs}. The lattice vector indices we use are given by
$$
J = \{
1 \mid  
2, 3 \mid  
4:6 \mid  
7:10 \mid 
12:14, 16, 17 \mid 
20:25 \mid
32:37, 49 \mid
54, 56, 64, 66, 67, 82, 83, 84\}. 
$$
Here, the vertical lines correspond to the horizontal lines in Table~\ref{t:LattVecs} and identify the dimension that the lattice vector first appears. For each dimension $d$, we reshape the upper triangular part of the matrices $v_j v_j^t$ into vectors and stacking the vectors as columns of a  matrix $F_d \in \mathbb R^{\binom{d}{2} \times \binom{d}{2}}$. The matrices $F_d$ are the lower left $\binom{d}{2} \times \binom{d}{2}$ block of the following matrix, $F \in \mathbb R^{36 \times 36}$, where to reduce size we use the shorthand 
$\blacktriangle = 1$, 
$\blacktriangledown = -1$, and 
$\blacksquare = \frac{\kappa^2}{4}$, 

{\small
\[F = 
\left[
\arraycolsep=2pt
\begin{array}{r|rr|rrr|rrrr|rrrrr|rrrrrr|rrrrrrr|rrrrrrrr}
 0& 0& 0& 0& 0& 0& 0&  0&  0&  0& 0&  0&  0& 0& 0& 0&  0&  0&  0&  0&  0& 0&  0&  0&  0&  0&  0& 0& \blacktriangle&  \blacktriangle&  \blacktriangle&  \blacktriangle&  \blacktriangle&  \blacktriangle&  \blacktriangle& \blacktriangle\\
 0& 0& 0& 0& 0& 0& 0&  0&  0&  0& 0&  0&  0& 0& 0& 0&  0&  0&  0&  0&  0& 0&  0&  0&  0&  0&  0& 0& 0&  0&  0& \blacktriangledown& \blacktriangledown&  0&  0& 0\\
 0& 0& 0& 0& 0& 0& 0&  0&  0&  0& 0&  0&  0& 0& 0& 0&  0&  0&  0&  0&  0& 0&  0&  0&  0&  0&  0& 0& 0&  0& \blacktriangledown& \blacktriangledown&  0&  0&  0& 0\\
 0& 0& 0& 0& 0& 0& 0&  0&  0&  0& 0&  0&  0& 0& 0& 0&  0&  0&  0&  0&  0& 0&  0&  0&  0&  0&  0& 0& 0&  0&  0&  0&  0&  0& \blacktriangledown& 0\\
 0& 0& 0& 0& 0& 0& 0&  0&  0&  0& 0&  0&  0& 0& 0& 0&  0&  0&  0&  0&  0& 0&  0&  0&  0&  0&  0& 0& 0&  0&  0&  0&  0& \blacktriangledown& \blacktriangledown& 0\\
 0& 0& 0& 0& 0& 0& 0&  0&  0&  0& 0&  0&  0& 0& 0& 0&  0&  0&  0&  0&  0& 0&  0&  0&  0&  0&  0& 0& 0& \blacktriangledown&  0&  0&  0&  0&  0& 0\\
 0& 0& 0& 0& 0& 0& 0&  0&  0&  0& 0&  0&  0& 0& 0& 0&  0&  0&  0&  0&  0& 0&  0&  0&  0&  0&  0& 0& 0& \blacktriangledown& \blacktriangledown&  0&  0&  0&  0& 0\\
 0& 0& 0& 0& 0& 0& 0&  0&  0&  0& 0&  0&  0& 0& 0& 0&  0&  0&  0&  0&  0& 0&  0&  0&  0&  0&  0& 0& 0&  0&  0&  0&  0&  0&  0& \blacktriangle\\
 \hline
 0& 0& 0& 0& 0& 0& 0&  0&  0&  0& 0&  0&  0& 0& 0& 0&  0&  0&  0&  0&  0& \blacktriangle&  \blacktriangle&  \blacktriangle&  \blacktriangle&  \blacktriangle&  \blacktriangle& \blacktriangle& 0&  0&  0&  \blacktriangle&  \blacktriangle&  0&  0& 0\\
 0& 0& 0& 0& 0& 0& 0&  0&  0&  0& 0&  0&  0& 0& 0& 0&  0&  0&  0&  0&  0& 0&  0&  0&  0&  0& \blacktriangledown& 0& 0&  0&  0&  \blacktriangle&  0&  0&  0& 0\\
 0& 0& 0& 0& 0& 0& 0&  0&  0&  0& 0&  0&  0& 0& 0& 0&  0&  0&  0&  0&  0& 0&  0&  0& \blacktriangledown&  0&  0& 0& 0&  0&  0&  0&  0&  0&  0& 0\\
 0& 0& 0& 0& 0& 0& 0&  0&  0&  0& 0&  0&  0& 0& 0& 0&  0&  0&  0&  0&  0& 0&  0&  0&  0&  \blacktriangle&  \blacktriangle& 0& 0&  0&  0&  0&  0&  0&  0& 0\\
 0& 0& 0& 0& 0& 0& 0&  0&  0&  0& 0&  0&  0& 0& 0& 0&  0&  0&  0&  0&  0& 0& \blacktriangledown&  0&  0&  0&  0& 0& 0&  0&  0&  0&  0&  0&  0& 0\\
 0& 0& 0& 0& 0& 0& 0&  0&  0&  0& 0&  0&  0& 0& 0& 0&  0&  0&  0&  0&  0& 0& \blacktriangledown& \blacktriangledown& \blacktriangledown& \blacktriangledown& \blacktriangledown& 0& 0&  0&  0&  0&  0&  0&  0& 0\\
 0& 0& 0& 0& 0& 0& 0&  0&  0&  0& 0&  0&  0& 0& 0& 0&  0&  0&  0&  0&  0& 0&  0&  0&  0&  0&  0& \blacktriangle& 0&  0&  0&  0&  0&  0&  0& 0\\
 \hline
 0& 0& 0& 0& 0& 0& 0&  0&  0&  0& 0&  0&  0& 0& 0& \blacktriangle&  \blacktriangle&  \blacktriangle&  \blacktriangle&  \blacktriangle&  \blacktriangle& 0&  0&  0&  0&  0&  \blacktriangle& 0& 0&  0&  \blacktriangle&  \blacktriangle&  0&  0&  0& 0\\
 0& 0& 0& 0& 0& 0& 0&  0&  0&  0& 0&  0&  0& 0& 0& 0& \blacktriangledown&  0& \blacktriangledown&  0&  0& 0&  0&  0&  0&  0&  0& 0& 0&  0&  0&  0&  0&  0&  0& 0\\
 0& 0& 0& 0& 0& 0& 0&  0&  0&  0& 0&  0&  0& 0& 0& 0& \blacktriangledown& \blacktriangledown& \blacktriangledown& \blacktriangledown& \blacktriangledown& 0&  0&  0&  0&  0& \blacktriangledown& 0& 0&  0&  0&  0&  0&  0&  0& 0\\
 0& 0& 0& 0& 0& 0& 0&  0&  0&  0& 0&  0&  0& 0& 0& 0&  \blacktriangle&  0&  0&  0&  0& 0&  0&  0&  0&  0&  0& 0& 0&  0&  0&  0&  0&  0&  0& 0\\
 0& 0& 0& 0& 0& 0& 0&  0&  0&  0& 0&  0&  0& 0& 0& 0&  0&  0&  0&  \blacktriangle&  \blacktriangle& 0&  0&  0&  0&  0&  \blacktriangle& 0& 0&  0&  \blacktriangle&  0&  0&  0&  0& 0\\
 0& 0& 0& 0& 0& 0& 0&  0&  0&  0& 0&  0&  0& 0& 0& 0&  0&  0&  0&  \blacktriangle&  0& 0&  0&  0&  0&  0&  0& 0& 0&  0&  0&  0&  0&  0&  0& 0\\
 \hline
 0& 0& 0& 0& 0& 0& 0&  0&  0&  0& \blacktriangle&  \blacktriangle&  \blacktriangle& \blacktriangle& \blacktriangle& 0&  \blacktriangle&  0&  \blacktriangle&  0&  0& 0&  0&  0&  \blacktriangle&  0&  0& 0& 0&  0&  0&  0&  0&  0&  \blacktriangle& 0\\
 0& 0& 0& 0& 0& 0& 0&  0&  0&  0& 0&  0&  \blacktriangle& 0& 0& 0&  \blacktriangle&  0&  \blacktriangle&  0&  0& 0&  0&  0&  0&  0&  0& 0& 0&  0&  0&  0&  0&  0&  \blacktriangle& 0\\
 0& 0& 0& 0& 0& 0& 0&  0&  0&  0& 0& \blacktriangledown& \blacktriangledown& 0& 0& 0& \blacktriangledown&  0&  0&  0&  0& 0&  0&  0&  0&  0&  0& 0& 0&  0&  0&  0&  0&  0&  0& 0\\
 0& 0& 0& 0& 0& 0& 0&  0&  0&  0& 0&  0&  0& \blacktriangle& \blacktriangle& 0&  0&  0&  0&  0&  0& 0&  0&  0&  \blacktriangle&  0&  0& 0& 0&  0&  0&  0&  0&  0&  0& 0\\
 0& 0& 0& 0& 0& 0& 0&  0&  0&  0& 0&  0&  0& \blacktriangle& 0& 0&  0&  0&  0&  0&  0& 0&  0&  0&  0&  0&  0& 0& 0&  0&  0&  0&  0&  0&  0& 0\\
 \hline
 0& 0& 0& 0& 0& 0& \blacktriangle&  \blacktriangle&  \blacktriangle&  \blacktriangle& 0&  0&  \blacktriangle& 0& 0& 0&  \blacktriangle&  \blacktriangle&  \blacktriangle&  \blacktriangle&  \blacktriangle& 0&  0&  0&  0&  \blacktriangle&  \blacktriangle& 0& 0&  0&  0&  0&  0&  \blacktriangle&  \blacktriangle& 0\\
 0& 0& 0& 0& 0& 0& 0&  0& \blacktriangledown&  0& 0&  0& \blacktriangledown& 0& 0& 0& \blacktriangledown&  0&  0&  0&  0& 0&  0&  0&  0&  0&  0& 0& 0&  0&  0&  0&  0&  0&  0& 0\\
 0& 0& 0& 0& 0& 0& 0& \blacktriangledown& \blacktriangledown& \blacktriangledown& 0&  0&  0& 0& 0& 0&  0&  0&  0& \blacktriangledown& \blacktriangledown& 0&  0&  0&  0& \blacktriangledown& \blacktriangledown& 0& 0&  0&  0&  0&  0&  0&  0& 0\\
 0& 0& 0& 0& 0& 0& 0&  0&  0& \blacktriangledown& 0&  0&  0& 0& 0& 0&  0&  0&  0& \blacktriangledown&  0& 0&  0&  0&  0&  0&  0& 0& 0&  0&  0&  0&  0&  0&  0& 0\\
 \hline
 0& 0& 0& \blacktriangle& \blacktriangle& \blacktriangle& 0&  0&  \blacktriangle&  0& 0&  \blacktriangle&  \blacktriangle& 0& 0& 0&  \blacktriangle&  0&  0&  0&  0& 0&  \blacktriangle&  0&  0&  0&  0& 0& 0&  \blacktriangle&  0&  0&  0&  0&  0& 0\\
 0& 0& 0& 0& \blacktriangle& \blacktriangle& 0&  0&  \blacktriangle&  0& 0&  0&  0& 0& 0& 0&  0&  0&  0&  0&  0& 0&  \blacktriangle&  0&  0&  0&  0& 0& 0&  \blacktriangle&  0&  0&  0&  0&  0& 0\\
 0& 0& 0& 0& 0& \blacktriangle& 0&  0&  0&  0& 0&  0&  0& 0& 0& 0&  0&  0&  0&  0&  0& 0&  0&  0&  0&  0&  0& 0& 0&  0&  0&  0&  0&  0&  0& 0\\
 \hline
 0& \blacktriangle& \blacktriangle& 0& \blacktriangle& \blacktriangle& 0&  \blacktriangle&  \blacktriangle&  \blacktriangle& 0&  0&  0& \blacktriangle& \blacktriangle& 0&  0&  0&  0&  \blacktriangle&  \blacktriangle& 0&  \blacktriangle&  \blacktriangle&  \blacktriangle&  \blacktriangle&  \blacktriangle& 0& 0&  \blacktriangle&  \blacktriangle&  0&  0&  0&  0& 0\\
 0& 0& \blacktriangle& 0& 0& \blacktriangle& 0&  0&  0&  \blacktriangle& 0&  0&  0& \blacktriangle& 0& 0&  0&  0&  0&  \blacktriangle&  0& 0&  0&  0&  0&  0&  0& 0& 0&  0&  0&  0&  0&  0&  0& 0\\
 \hline
\blacksquare& 0& \blacktriangle& 0& 0& \blacktriangle& 0&  0&  0&  \blacktriangle& 0&  0&  0& \blacktriangle& 0& 0&  0&  0&  0&  \blacktriangle&  0& 0&  0&  0&  0&  0&  0& \blacktriangle& 0&  0&  0&  0&  0&  0&  0& \blacktriangle
\end{array}
\right].
\]}
We will show that the matrices $F_d$ have non-zero determinant, and hence the outer products are linearly independent. If the matrix $F$ has non-zero determinant, then  $F_d$ for each $d\leq 8$ also has non-zero determinant. Observing that the upper left submatrix blocks of $F$ are zero, we see that the determinant of $F$ is the product of the lower-left to upper-right diagonal sub-blocks of the matrix $F$. Judiciously choosing the minors in the Laplace expansion for the determinant, we obtain $|\det F| = |\det F_d| = \frac{\kappa^2}{4} \neq 0$. 
\end{proof}

\subsection{Stationarity of tori} \label{s:LocOpt}
Here, for each $k \geq 1$ and $2\leq d \leq 8$, we give a vector $c = c(k,d)  \in \mathbb R^{\frac m 2}$, that satisfies the stationarity condition given in  Theorem~\ref{t:NeccCond}.
We first observe that such a vector $c$, if one exists, is not unique for $d = 4,5,6,7,8$, as the following argument shows. Reshaping the symmetric matrices $M_j$ as defined in \eqref{e:M} into vectors of length $\binom{d}{2}$, the condition for stationarity in Theorem~\ref{t:NeccCond} is that a non-negative linear combination gives zero. Of course, if the number of vectors, $\frac{m}{2}$,  exceeds $\binom{d}{2}$, \ie, $m > d(d+1)$, then the columns are linearly dependent. Looking at Table~\ref{t:EigTable}, this is the case for d = 4,5,6,7,8.

\begin{table}[t]
\begin{tabular}{c| c c c c c c c}
$d$ & 2 & 3 & 4 & 5 & 6 & 7 & 8 \\
\hline 
$a_{k,d}$ & 
$\frac{2 \kappa^2-4}{\kappa^2}$ & 
$\frac{3 \kappa^2 -8}{ \kappa^2}$ & 
$4 \frac{\kappa^2 - 4 }{\kappa^2 }$ & 
$6 \frac{ \kappa^2 - 4}{\kappa^2}$ & 
$8 \frac{ \kappa^2 - 5} {\kappa^2}$ & 
$4 \frac{ 3 \kappa^2 - 16} {\kappa^2}$ & 
$18 \frac{ \kappa^2 - 6} {\kappa^2}$
\\
$b_{k,d}$ & 
$\cdot$ & 
$\frac{2 \kappa^2-4} {\kappa^2}$ & 
$2 \frac{\kappa^2 - 4}{\kappa^2 }$ &
$2 \frac{ \kappa^2 - 3} {\kappa^2}$ & 
$2 \frac{ \kappa^2 - 4} {\kappa^2}$ & 
$2 \frac{ \kappa^2 - 4} {\kappa^2}$ & 
$\frac{ 2 \kappa^2 - 9} {\kappa^2}$
\end{tabular}
\caption{The values of $a_{k,d}$ and $b_{k,d}$ used in the definition of the vector $c^\circ$. See Section~\ref{s:LocOpt}.}
\label{t:c} 
\end{table}

For $k=1,2$, for every $2\leq d \leq 8$, 
define 
$c^\circ =( 1, \cdots, 1)\in \mathbb R^{\frac m 2}$.
For  $k\geq 3$ and $2\leq d \leq 8$, 
define the vector $c^\circ \in \mathbb R^{\frac m 2}$ by 
$$
c^\circ_i = \begin{cases}
a_{k,d} & i =1  \\ 
b_{k,d} & i \in I \\ 
1 & \textrm{otherwise}, 
\end{cases}
$$
where the constants $a_{k,d}$ and $b_{k,d}$ are specified in Table~\ref{t:c} and the index set $I$ is defined 
\begin{align*}
I :=  
\{ &4 \mid  7 \mid 
12 : 15 \mid 
20 : 23 \mid 
33 : 42 \mid   
55 : 70\}.
\end{align*}
The indices in $I$ correspond to the lattice vectors in Table~\ref{t:LattVecs}, where the indices in $I$ are displayed in red. The vertical lines here correspond to the horizontal lines in Table~\ref{t:LattVecs} and identify the dimension that the lattice vector first appears. 

\begin{no} \label{no:Stat}
For every $k\geq 1$, and  $2\leq d \leq 8$, 
the vector $c^\circ \in \mathbb R^{\frac m 2}$ satisfies the stationarity condition given in  Theorem~\ref{t:NeccCond}. 
\end{no}

It is straightforward to check that Numerical Observation~\ref{no:Stat} holds. Sage code that symbolically verifies the claim is provided in \cite{github}. Comments on how we first identified $c^\circ \in \mathbb R^{\frac m 2}$ are made in Appendix~\ref{s:c}.

As an example, we verify Numerical Observation~\ref{no:Stat} in dimension $d=3$.  We have
$G^\circ_{k,d} = 
\begin{pmatrix}
2 \kappa^2 & - \kappa^2 & 0 \\
- \kappa^2 & 2 \kappa^2 & -4 \\
0 & -4 & 8 \kappa^2 
\end{pmatrix}$
so that 
$$
(G^\circ_{k,d})^{-1} = \frac{1}{8 (3 \kappa^2 - 4)} 
\begin{pmatrix}
16 \frac{\kappa^2 - 1}{\kappa^2} & 8 & 4 \\
8 & 16 & 8 \\
4 & 8 & 3 \kappa^2 
\end{pmatrix},
$$
and 
$\det G^\circ_{k,d} = 8 \kappa^2 (3 \kappa^2 - 4)$. 
We also have 
$\Lambda^\circ_{k,3} = 
4 \pi^2 \left( \frac{\kappa^{4}}{3 \kappa^{2}- 4} \right)^{\frac{1}{3}} = 8 \pi^2 \kappa^2 \left( \det G^\circ_{k,d} \right)^{-1/d}$ 
(see Table~\ref{t:EigTable}),  
and, from Table~\ref{t:LattVecs} and Table~\ref{t:c},
\begin{align*}
c_1^\circ & = \frac{3 \kappa^2 -8}{ \kappa^2}    
&&v_1 v_1^t = 
\begin{pmatrix}
0 & 0 & 0 \\
0 & 0 & 0 \\
0 & 0 & \frac{\kappa^2}{4} 
\end{pmatrix} \\
c_2^\circ & = 1    
&&v_2 v_2^t = 
\begin{pmatrix}
0 & 0 & 0 \\
0 & 1 & 0 \\
0 & 0 & 0
\end{pmatrix} \\
c_3^\circ & = 1    
&&v_3 v_3^t = 
\begin{pmatrix}
0 & 0 & 0 \\
0 & 1 & 1 \\
0 & 1 & 1
\end{pmatrix} \\
c_4^\circ &=  \frac{2 \kappa^2-4} {\kappa^2}
&& v_4 v_4^t = 
\begin{pmatrix}
1 & 0 & 0 \\
0 & 0 & 0 \\
0 & 0 & 0
\end{pmatrix} \\
c_5^\circ & = 1    
&&v_5 v_5^t = 
\begin{pmatrix}
1 & 1 & 0 \\
1 & 1 & 0 \\
0 & 0 & 0
\end{pmatrix} \\
c_6^\circ & = 1    
&&v_6 v_6^t = 
\begin{pmatrix}
1 & 1 & 1 \\
1 & 1 & 1 \\
1 & 1 & 1
\end{pmatrix}. 
\end{align*}
We can then compute 
\begin{align*} 
\sum_{i=1}^6 c_i^\circ M_i 
&= 
- \frac{\Lambda^\circ_{k,3}}{d} \left( \sum_{i=1}^6 c_i^\circ  \right) (G^\circ_{k,d})^{-1} + 
4 \pi^2 \left( \det G^\circ_{k,d} \right)^{- \frac 1 d} 
\sum_{i=1}^6 c_i^\circ v_i v_i^t. 
\end{align*}
Since $\sum_{i=1}^6 c_i^\circ =3 \left( \frac{3 \kappa^2 - 4 }{\kappa^2}  \right) $ and 
$$\sum_{i=1}^6 c_i^\circ v_i v_i^t = 
\begin{pmatrix}
4(\kappa^2 - 1)/\kappa^2 & 2 & 1 \\
2 & 4 & 2 \\
1 & 2 & 3 \kappa^2/4
\end{pmatrix}, 
$$
we obtain $\sum_{i=1}^6 c_i^\circ M_i  =  0$.

\subsection{Degeneracy of flat tori as \texorpdfstring{$k\to \infty$}{k tends to infinity}} \label{s:Degeneracy}
For fixed $k \in \mathbb N$ and $d\leq8$, denote the eigenvalues of the normalized Gram matrix
$\frac{G_{k,d}^\circ}{\textrm{det}\left( G_{k,d}^\circ \right)^{\frac{1}{d}}}$ by 
$\mu^k_1 \leq \mu^k_2 \leq \cdots \leq \mu^k_d$. 
In Figure~\ref{f:degenTori}, we plot 
$k$ vs $\mu^k_1,\ldots, \mu^k_d$ for $d=2,\ldots, 8$. 
From  Figure~\ref{f:degenTori}, we hypothesize that, in each dimension, for large $k$,
$$
\mu_i^k
 \sim c_i k^{2 \cdot p_i/d}, 
$$
for constants $c_i, p_i$ that are independent of $k$. In particular, $p_1 = -(d-1)$ and  $p_2=\cdots = p_d = 1$, which necessarily satisfy $\sum_{i=1}^d p_i = 0$.

\begin{figure}
\centering
\includegraphics[width=0.8\textwidth,trim={0 1.5cm 0 1.5cm},clip]{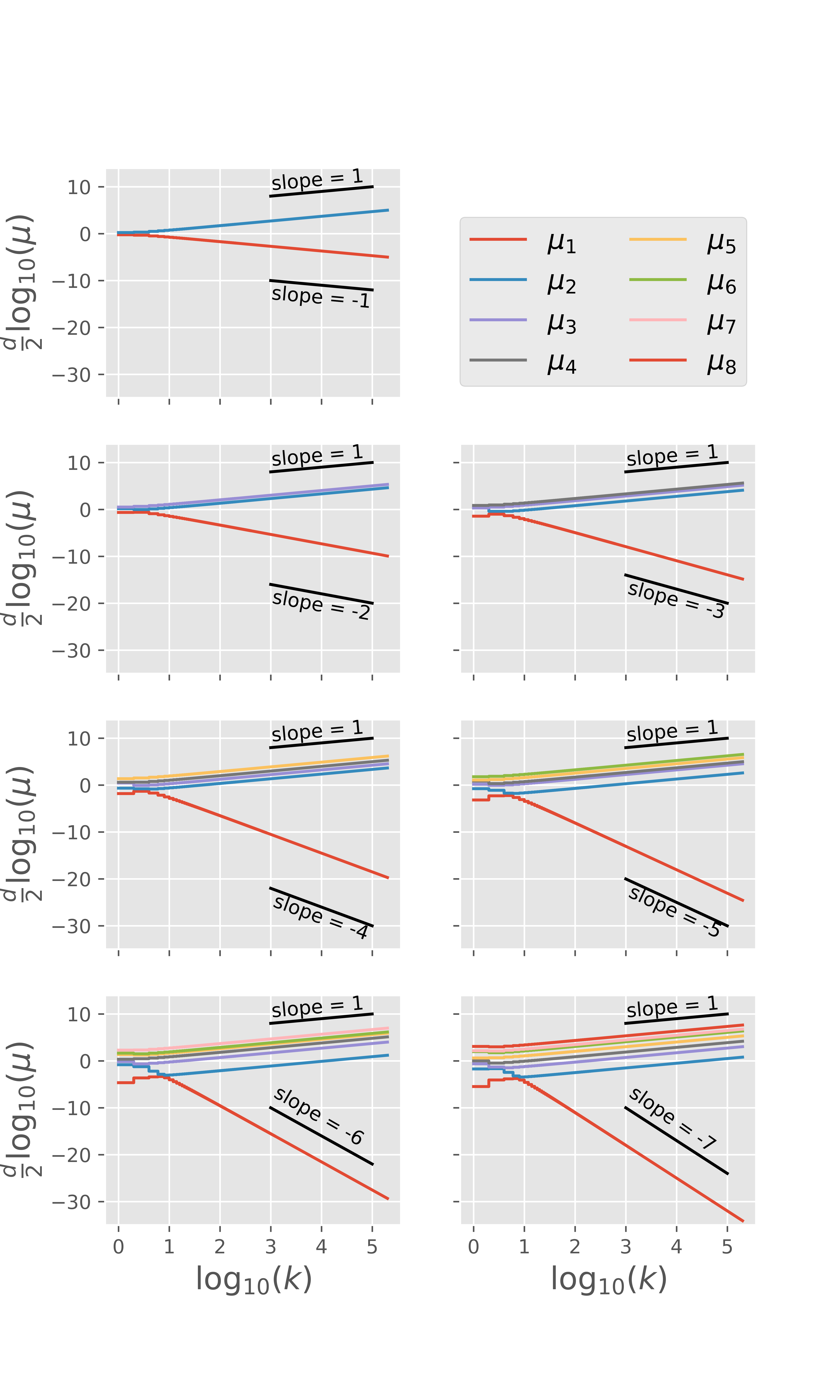}
\caption{\label{f:degenTori} Log-log plots of $k$ vs the eigevalues, $\mu$, of the Gram matrix, $G_{k,d}^\circ$ for dimensions $d=2,\ldots,8$.}
\end{figure}

In Figure \ref{f:ellipsoid}, we give further geometric interpretation of $v^t G^\circ_{k,d} v = 2 \kappa^2$ in dimensions $d=2$ and $3$. For $k = 1,3,5,7$, we plot the ellipses/ellipsoids corresponding to the Gram matrix, as well as the $k$-th shortest lattice vectors.  In $d=2$ dimensions, the ellipses intersect  six lattice points and elongate as $k$ increases in one direction. In $d=3$ dimensions, the ellipsoids intersect 12 lattice points and again elongate in one direction.   In both cases, the elongation in one direction corresponds to the first eigenvalue of the Gram matrix scaling as $\mu_1(k) \sim c_1 k^{- 2(d-1)/d}$ and the other eigenvalues scaling as $\mu_i(k)  \sim c_i k^{2/d}$.

\begin{figure}
\centering
\includegraphics[width=0.8\textwidth]{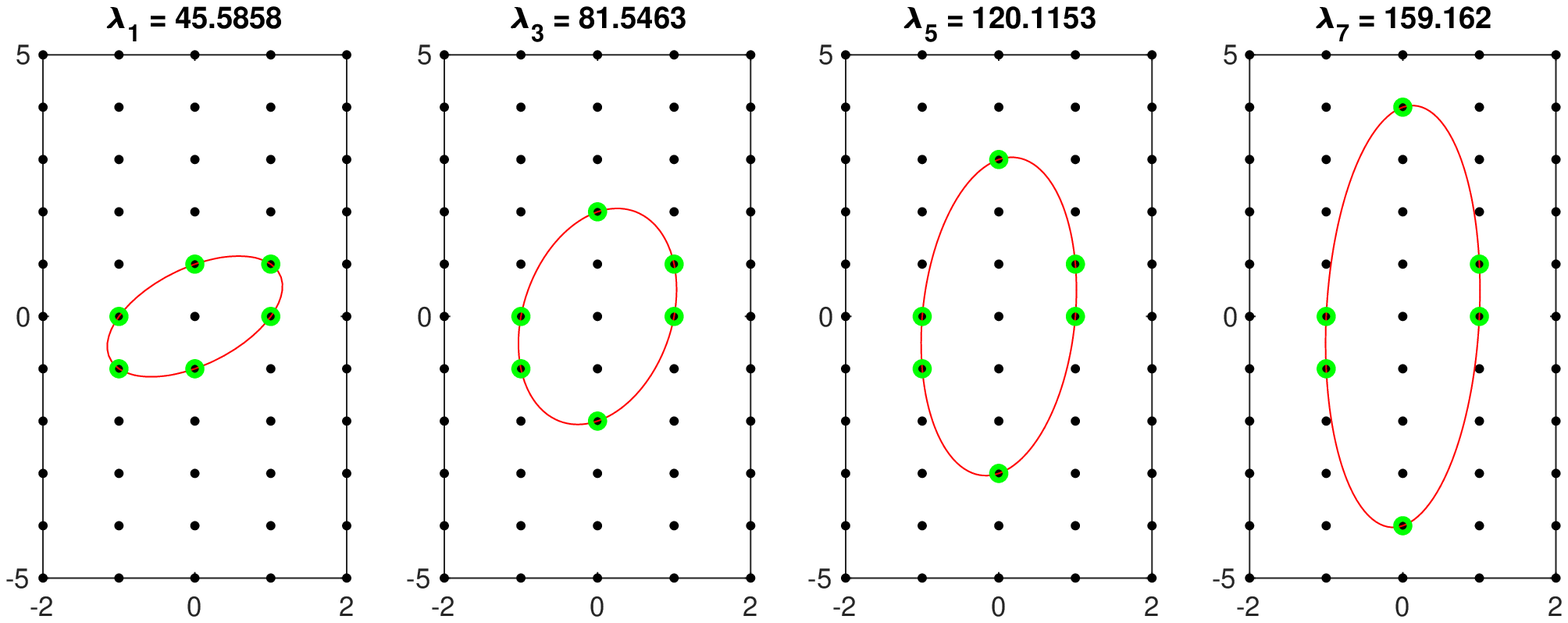}
\includegraphics[width=1\textwidth]{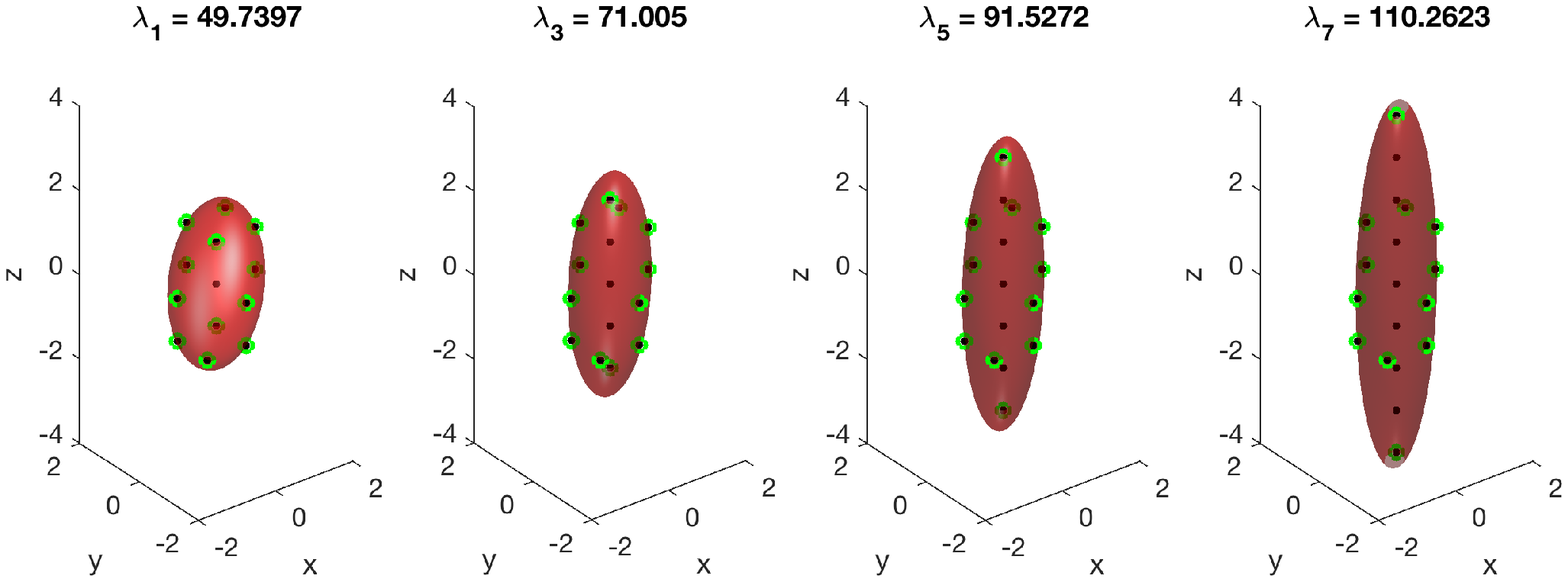}
\caption{\label{f:ellipsoid} The ellipses ($d=2$) and ellipsoids ($d=3$) corresponding to the Gram matrix, $G_{k,d}^\circ$  for $k=1,3,5,7$. The intersecting lattice points are indicated.}
\end{figure}

Finally, we conclude with a discussion of the successive minima. Recall that, for $1 \leq i \leq d$, the \emph{$i$-th successive minimum} of a lattice with basis $B$ is defined by 
$$
\gamma_i(B) = \min \{ \| v_i \| \colon \exists \textrm{ lin. ind. }  v_1\ldots, v_i \in B \mathbb Z^d  \textrm{ with } \| v_1\| \leq \ldots  \leq \|v_i\| \}. 
$$
That is, $\gamma_i$ is the smallest number $\gamma$, such that the ellipsoid $\{ \| Bx\| \leq \gamma \}$  contains $i$ linearly independent vectors (see, \eg, \cite{schurmann2009computational}). 
For each $k\geq 1$, $d\leq 8$, we have $\gamma_1( (B^\circ_{k,d})^{-t}) = 2 \sqrt{2}$ is attained by the vector $(0,\cdots,1)$ and 
$\gamma_2(  (B^\circ_{k,d})^{-t}) = \cdots = \gamma_d( (B^\circ_{k,d})^{-t}) = \sqrt{2} \kappa$. 
In particular, since  the injectivity radius of a flat torus, $T_B$, satisfies  $\textrm{inj} (T_B) = \gamma_1(B) \asymp \gamma_d(B^{-t})^{-1}$, we have that   $
\textrm{inj} (T_{k,d}^\circ)  \asymp k^{-1}$. 
If we scale $T_{k,d}^\circ$ by $\alpha = \textrm{vol}(T_{k,d}^\circ)^{- \frac{1}{d}} = | \det G_{k,d}^\circ |^{ \frac{1}{2d}} \asymp k^{ \frac{d-1}{d}}$ (see Table~\ref{t:EigTable}), 
we obtain that
$\textrm{vol}( \alpha T_{k,d}^\circ) = 1$.
We then compute 
$\textrm{inj} ( \alpha T_{k,d}^\circ)  
= \gamma_1( \alpha B) 
= \alpha \gamma_1(  B) 
\asymp  \alpha \gamma_d( B^{-t})^{-1} 
= k^{-\frac{1}{d}}$, 
which is consistent with \cite[Thm. 1.2]{Lagac__2019}.

\subsection*{Acknowledgments} We would like to thank Jean Lagac\'e and Lenny Fukshansky for useful conversations.

\printbibliography

\clearpage
\appendix
\section{Numerical methods} 
\label{s:NumMeth}
In this appendix, we describe a numerical method for  approximating solutions to the optimization problem in \eqref{e:OptProb} and for generating  a vector $c  \in \mathbb R^{\frac m 2}$ that satisfies the stationarity condition given in  Theorem~\ref{t:NeccCond}.

\subsection{Optimization method.}
In  Section~\ref{s:proof:OptLattice},
we explained how we can compute Laplacian eigenvalues of given tori. Here we describe an optimization method for generating those lattices. 
The optimization problem in \eqref{e:OptProb} can be trivially rewritten as 
\begin{subequations} \label{e:Opt2}
\begin{align}
\Lambda_{k,d}^\star = \max_{\alpha,B} & \ \alpha \\
\textrm{s.t.} & \ \Lambda_j (T_B) \geq \alpha, \qquad j \geq k.
\end{align}
\end{subequations}

Our strategy for solving \eqref{e:Opt2} is to successively solve its linearization. Writing 
$$
B^{-t} = (I + \varepsilon) B_0^{-t}
$$
for some fixed $B_0 \in GL(d,\mathbb R)$ and a  matrix $\varepsilon \in GL(d,\mathbb R)$ with small norm, we compute the first-order approximations   
\begin{align*}
\det B^{-t} &= \det B_0^{-t} \det( I + \varepsilon) \\
& = \det B_0^{-t} \left(1 + \langle I, \varepsilon\rangle_F + o(\| \varepsilon \|)  \right) 
\end{align*}
and 
\begin{align*}
\lambda_k (T_B)  &= 4 \pi^2 \| B^{-t} v_k \|^2 \\
& = 4 \pi^2 \left(  v_k^t B_0^{-1} (I + \varepsilon^t) (I + \varepsilon)  B_0^{-t} v_k \right)  \\
& =  4 \pi^2 \| B_0^{-t} v_k \|^2   + 8 \pi^2 \langle  (B_0^{-t} v_k)  (B_0^{-t}v_k)^t  , \varepsilon \rangle_F + o(\| \varepsilon\|). 
\end{align*}
Combining these, and assuming $\det B_0^{-t} > 0$, we obtain 
{\small
\begin{align*}
\Lambda_{k,d} (B) & = \lambda_k (T_B) |\det B^{-t} |^{ - \frac{2}{d}} \\
& = \left( 4 \pi^2 \| B_0^{-t} v_k \|^2   + 8 \pi^2 \langle  (B_0^{-t} v_k)  (B_0^{-t}v_k)^t  , \varepsilon \rangle_F + o(\| \varepsilon\|) \right) \det(B_0^{-t})^{- \frac{2}{d}} \left(1 + \langle I , \varepsilon \rangle_F + o(\| \varepsilon \|)  \right)^{- \frac{2}{d}} \\
& = \Lambda_k(B_0) + \langle  \Sigma_k , \varepsilon \rangle_F + o(\| \varepsilon \|), 
\end{align*} }
where
$$
\Sigma_k := 8 \pi^2 \det(B_0^{-t})^{- \frac{2}{d}}  (B_0^{-t} v_k)  (B_0^{-t}v_k)^t - \frac{2}{d} \Lambda_{k,d} (B_0) I.
$$
A linearization of the optimization problem in \eqref{e:Opt2}  is then 
\begin{subequations} 
\label{e:LinOpt}
\begin{align}
\max_{\alpha,\varepsilon} & \ \alpha \\
\label{e:LinOptb}
\textrm{s.t.} & \  \Lambda_{j,d}(B_0) + \langle \Sigma_j, \varepsilon \rangle_F \geq \alpha , \qquad j \geq k. 
\end{align}
Additionally, for $\beta>0$, we add the diagonally dominant constraints 
\begin{align}
- \beta &\leq \varepsilon_{i,i} \leq \beta \qquad i \in [d] \\
- \frac{\beta}{d-1} &\leq \varepsilon_{i,j} \leq \frac{\beta}{d-1}  \qquad i \neq j
\end{align}
\end{subequations}
which ensure $\| \varepsilon \| < 2 \beta$. 
We retain only a finite number of constraints in \eqref{e:LinOptb} by considering only the $j = k, \ldots, k+ K + 1$ for some integer $K>1$ shortest lattice vectors. 
This linearization procedure is similar to that appearing in \cite{Marcotte_2013} for the closest packing problem. 

The linear optimization problem \eqref{e:LinOpt}, which depends on the parameter $\beta$, is then solved using the Gurobi linear programming library \cite{Gurobi} repeatedly until $\| B - B_0\|$ falls below a specified tolerance.  
The parameter $\beta$ is treated as a trust-region parameter and adaptively set at each iteration to ensure that the linearization of $\Lambda_{k,d}\left( B \right)$ is faithful. 

In these numerical computations, floating point arithmetic was performed to find the maximal lattice $B_{k,d}$. We then formed the Gram matrix for the dual lattice, $B^{-1}_{k,d} B^{-t}_{k,d} $ and observed  numerically that all of the elements are multiples of the smallest nonzero element of the Gram matrix, suggesting that the Gram matrix can be rescaled as an integer matrix.  We then used the 
Lenstra-Lenstra-Lov\'asz (LLL) lattice basis reduction algorithm \cite{Lenstra_1982}
to simplify the matrix and row/column permutations to obtain the laminated structure of $\mathcal{G}_k$.

\subsection{Numerical method for the stationarity condition.} \label{s:c}
Here, we explain how, in Section~\ref{s:LocOpt}, we computed a vector $c^\circ = c^\circ(k,d)  \in \mathbb R^{\frac m 2}$, that satisfies the stationarity condition given in  Theorem~\ref{t:NeccCond}. As explained in  Section~\ref{s:LocOpt}, the vector is not unique, so it is challenging to derive a general formula for $c^\circ$ from an (arbitrarily computed)  solution for various $k,d$. To overcome this obstacle, we specify an addition condition that gives uniqueness. 
For fixed $k\geq 1$, $2\leq d \leq 8$, $m$ the multiplicity of the eigenvalue, and $M_j\in \mathbb R^{d\times d}$, $j = 1,\ldots, \frac{m}{2}$ as defined in \eqref{e:M}, we consider the quadratic optimization problem 
\begin{subequations}
\begin{align}
\min_{c \in \mathbb R^{\frac m 2}}  \ & \|c\|_2^2 \\
\textrm{such that} \ &  c \geq 0 \\ 
& \sum_{j=1}^{m/2} c_j M_j = 0 \\ 
& c_1 = 1. 
\end{align}
\end{subequations}
which asks for the shortest vector $c$ (in the $\ell^2$ sense) that satisfies the desired properties. For each dimension $d$, we solved this problem for small values of $k$, and were able to deduce the general formula, yielding $c^\circ  \in \mathbb R^{\frac m 2}$ as given in Section~\ref{s:LocOpt}. 

\end{document}